\documentclass[11pt]{article}
\usepackage[a4paper,margin=28mm]{geometry}
\usepackage{amsmath,amssymb,amsthm,mathtools}
\usepackage{enumitem}
\usepackage{microtype}
\usepackage[hidelinks]{hyperref}
\usepackage{booktabs}
\usepackage{url}

\newtheorem{theorem}{Theorem}[section]
\newtheorem{lemma}[theorem]{Lemma}

\theoremstyle{definition}
\newtheorem{definition}[theorem]{Definition}

\theoremstyle{remark}
\newtheorem{remark}[theorem]{Remark}

\newcommand{\ceil}[1]{\left\lceil #1\right\rceil}
\newcommand{\floor}[1]{\left\lfloor #1\right\rfloor}

\title{Iterated Distinct Absolute Differences of Integer Compositions}
\author{Felix Huber\thanks{ORCID: \href{https://orcid.org/0009-0005-1568-1579}{0009-0005-1568-1579}}}
\date{August 2026}

\begin{document}
\maketitle



\begin{abstract}
Starting from an integer composition, form its consecutive absolute differences, provided that they are nonzero and pairwise distinct, and then permute these differences arbitrarily before repeating the operation. The depth of the composition is the maximum possible number of successive iterations. We determine the least integer admitting a composition of any prescribed depth. If $a(n)$ is the least positive integer having a composition of depth $n$, then
\[
a(n)=n+1+\ceil{\frac{n(n+1)}4}+\floor{\frac n2}.
\]
We also prove that every integer $k\ge a(n)$ has a composition of depth at least $n$. Consequently, if $d(k)$ denotes the maximum depth of a composition of $k$, then
\[
d(k)=\max\{n\ge0:a(n)\le k\}.
\]
Finally, we classify and enumerate all compositions attaining the minimum $a(n)$. Their number is given by four factorial formulas according to $n$ modulo $4$.
\end{abstract}

\section{Introduction}

For a finite sequence $x=(x_0,\ldots,x_m)$, its consecutive absolute differences are
\[
\Delta(x)=(|x_1-x_0|,\ldots,|x_m-x_{m-1}|).
\]
The classical theory of graceful permutations asks for permutations whose consecutive absolute differences are all distinct; equivalently, these are graceful labelings of paths. Graceful labeling was introduced by Rosa~\cite{Rosa1967}; for a broad survey see Gallian~\cite{Gallian2025}. Graceful permutations and their enumeration have been studied, for example, by Adamaszek~\cite{Adamaszek2013}. The alternating Walecki arrangement is a standard construction producing all differences $1,\ldots,m-1$ from an ordering of $1,\ldots,m$; see also the discussion in~\cite{Ollis2020}.

Here we study a different extremal problem. The entries of the starting sequence are the positive parts of a composition, repeated values are allowed, and after each difference step the resulting distinct positive values may be reordered arbitrarily. We ask how small the sum can be if a prescribed number of iterations is required.

Our first result gives an exact quadratic quasipolynomial for this minimum. The second shows that the attainable sums have no gaps: once depth $n$ becomes possible, it remains possible for every larger sum. The third classifies all minimizers and gives exact factorial enumeration formulas.

\section{Definitions and main results}

\begin{definition}
A \emph{composition of $k$} is a finite sequence of positive integers whose sum is $k$. Let
\[
C_0=(c_0,\ldots,c_m)
\]
be a composition. One iteration is allowed if the $m$ numbers
\[
|c_1-c_0|,\ldots,|c_m-c_{m-1}|
\]
are nonzero and pairwise distinct. In that case they may be arranged in any order to form $C_1$. The procedure is then repeated. Formally, the \emph{depth} of $C_0$ is
\[
\operatorname{depth}(C_0)=\max\{t\ge0:\text{there exist valid }C_0,C_1,\ldots,C_t\}.
\]
\end{definition}

Every iteration reduces the length by one, so a composition with $m+1$ parts has depth at most $m$.

Let $a(n)$ be the least positive integer $k$ for which some composition of $k$ has depth $n$, and let $d(k)$ be the maximum depth of a composition of $k$.

Define
\[
M(n)=\ceil{\frac{n(n+1)}4}+\floor{\frac n2}.
\]

\begin{theorem}[Minimum sum]\label{thm:min}
For every $n\ge0$,
\[
\boxed{a(n)=n+1+M(n)
=n+1+\ceil{\frac{n(n+1)}4}+\floor{\frac n2}.}
\]
Equivalently,
\[
\begin{aligned}
a(4q)&=4q^2+7q+1,\\
a(4q+1)&=4q^2+9q+3,\\
a(4q+2)&=4q^2+11q+6,\\
a(4q+3)&=4q^2+13q+8.
\end{aligned}
\]
\end{theorem}

\begin{theorem}[No gaps]\label{thm:nogaps}
For every $n\ge0$ and every $k\ge a(n)$, there exists a composition of $k$ having depth at least $n$. Consequently,
\[
\boxed{d(k)\ge n\iff k\ge a(n)}
\]
and
\[
\boxed{d(k)=\max\{n\ge0:a(n)\le k\}.}
\]
In particular, $d(k)$ is nondecreasing.
\end{theorem}

To state the enumeration result, call a nonnegative sequence $y=(y_0,\ldots,y_n)$ \emph{normalized} if $\min_i y_i=0$. Let $N(n)$ denote the number of normalized sequences whose differences are $1,\ldots,n$ in some order and whose sum is $M(n)$.

\begin{theorem}[Enumeration of minimizers]\label{thm:count}
The number $N(n)$ is also the number of compositions of $a(n)$ having depth $n$. Moreover,
\[
\boxed{
N(n)=
\begin{cases}
1,&n=0,\\
3(2q)!,&n=4q,\ q\ge1,\\
2(2q+1)!,&n=4q+1,\\
4(2q+1)!,&n=4q+2,\\
2(q+1)(2q+1)!,&n=4q+3.
\end{cases}}
\]
Thus the first values, beginning with $n=0$, are
\[
1,2,4,2,6,12,24,24,72,240,480,720,2160,\ldots.
\]
\end{theorem}

\section{The lower bound}

The case $n=0$ is immediate: $(1)$ has depth $0$, so $a(0)=1$. Assume henceforth that $n\ge1$.

Suppose that
\[
x=(x_0,\ldots,x_m)
\]
has depth at least $n$. Since each iteration reduces the length by one, $m\ge n$. The first iteration is possible, so
\[
e_i=|x_i-x_{i-1}|,\qquad 1\le i\le m,
\]
are distinct positive integers.

Let
\[
c=\min_{0\le i\le m}x_i,
\qquad y_i=x_i-c,
\]
and put
\[
S=\sum_{i=1}^m e_i,
\qquad A=\sum_{i=0}^m y_i.
\]
Then $y_i\ge0$, at least one $y_i$ is zero, and
\[
S\ge1+2+\cdots+m=\frac{m(m+1)}2.
\]

For nonnegative integers $u,v$,
\[
u+v-|u-v|=2\min(u,v).
\]
Summing over all adjacent pairs gives
\begin{equation}\label{eq:Bidentity}
B:=2A-S
=y_0+y_m+2\sum_{i=1}^m\min(y_{i-1},y_i).
\end{equation}

\begin{lemma}\label{lem:Bbound}
For every such sequence,
\[
B\ge2\floor{\frac m2}.
\]
\end{lemma}

\begin{proof}
Split the sequence at all zero entries.

Consider first an inner block between two consecutive zeros, with $r$ edges:
\[
0,z_1,\ldots,z_{r-1},0,
\qquad z_i>0.
\]
The cases $r=1$ and $r=2$ are impossible: the former gives a zero edge, and the latter repeats the same positive edge. Its contribution to $B$ is
\[
2\sum_{i=1}^{r-2}\min(z_i,z_{i+1})\ge2(r-2).
\]
For $r\ge4$ this is at least $r$. For $r=3$, the contribution is $2\min(z_1,z_2)$. It is at least $3$ unless $\min(z_1,z_2)=1$, in which case it equals $2$, one less than the number of edges. Such a deficient block contains an edge of difference $1$, so at most one deficient block can occur.

An end block with $r$ edges has, after possible reversal, the form
\[
z_0,z_1,\ldots,z_{r-1},0,
\qquad z_i>0.
\]
Its contribution to $B$ is
\[
z_0+2\sum_{i=1}^{r-1}\min(z_{i-1},z_i)
\ge1+2(r-1)=2r-1\ge r.
\]
Thus every block contributes at least its number of edges, except possibly one deficient three-edge block, which contributes one less. Hence
\[
B\ge m-1.
\]
If $m$ is odd, this is exactly $2\floor{m/2}$.

Suppose now that $m$ is even and that equality $B=m-1$ holds. Then there is exactly one deficient three-edge block, every other inner block has four edges and attains equality, and every end block has one edge and attains equality. The edge count is therefore a sum of one $3$, some $4$'s, and at most two $1$'s. Since $m$ is even, exactly one one-edge end block must occur. Equality for that end block forces it to be $(1,0)$ or $(0,1)$, so it uses edge difference $1$. The deficient three-edge block also uses difference $1$, contradicting distinctness. Therefore $B\ge m$ when $m$ is even, completing the proof.
\end{proof}

By Lemma~\ref{lem:Bbound},
\[
A=\frac{S+B}{2}
\ge\ceil{\frac S2}+\floor{\frac m2}
\ge\ceil{\frac{m(m+1)}4}+\floor{\frac m2}=M(m).
\]
Since $c\ge1$,
\[
\sum_{i=0}^m x_i=A+(m+1)c\ge M(m)+m+1.
\]
The right-hand side is increasing in $m$, and $m\ge n$, so
\[
\sum_{i=0}^m x_i\ge M(n)+n+1.
\]
This proves
\[
a(n)\ge n+1+M(n).
\]

\section{Optimal constructions}

We now construct, for each $n\ge1$, a normalized sequence
\[
y=(y_0,\ldots,y_n)
\]
whose consecutive absolute differences are exactly $1,\ldots,n$ and whose sum is $M(n)$. Adding $1$ to every entry then gives a composition of sum $n+1+M(n)$.

\subsection*{Case $n=4q$}
For $q\ge1$, set $y_0=0$, $y_n=2$, and for $1\le i<n$ define
\[
y_i=
\begin{cases}
i+2,&i\equiv0\pmod2,\\
1,&i\equiv1\pmod4,\\
0,&i\equiv3\pmod4.
\end{cases}
\]
Thus
\[
(0,1,4,0,6,1,8,0,10,1,\ldots,n,0,2).
\]
Its differences are
\[
1,3,4,6,5,7,8,10,9,11,12,\ldots,n,2,
\]
which are exactly $1,\ldots,n$. Direct summation gives
\[
\sum_i y_i=4q^2+3q=M(4q).
\]

\subsection*{Case $n=4q+1$}
Set $y_0=0$, and for $1\le i\le n$ define
\[
y_i=
\begin{cases}
i+1,&i\equiv0\pmod2,\\
1,&i\equiv1\pmod4,\\
0,&i\equiv3\pmod4.
\end{cases}
\]
The sequence begins
\[
(0,1,3,0,5,1,7,0,9,1,\ldots).
\]
Its differences are
\[
1,2,3,5,4,6,7,9,8,10,11,\ldots,
\]
and its sum is
\[
4q^2+5q+1=M(4q+1).
\]

\subsection*{Case $n=4q+2$}
For $0\le i\le n$, set
\[
y_i=
\begin{cases}
i+1,&i\equiv1\pmod2,\\
1,&i\equiv0\pmod4,\\
0,&i\equiv2\pmod4.
\end{cases}
\]
Thus the sequence begins
\[
(1,2,0,4,1,6,0,8,1,10,0,\ldots).
\]
Its differences are
\[
1,2,4,3,5,6,8,7,9,10,\ldots,
\]
and
\[
\sum_i y_i=4q^2+7q+3=M(4q+2).
\]

\subsection*{Case $n=4q+3$}
Use the same rule as in the case $n=4q+1$, truncated at an index congruent to $3$ modulo $4$. Its differences are $1,\ldots,n$, and
\[
\sum_i y_i=4q^2+9q+4=M(4q+3).
\]

It remains to verify the depth. For every $m\ge1$, arrange $1,\ldots,m$ in Walecki order
\[
1,m,2,m-1,3,m-2,\ldots.
\]
Its consecutive absolute differences are $m-1,m-2,\ldots,1$. Therefore, after the first iteration, recursively apply the Walecki ordering to the current set $\{1,\ldots,m\}$. The successive difference sets are
\[
\{1,\ldots,n\},\{1,\ldots,n-1\},\ldots,\{1,2\},\{1\}.
\]
The constructed composition has depth at least $n$, and since it has exactly $n+1$ parts, its depth is exactly $n$. This proves Theorem~\ref{thm:min}.

\section{Every larger sum is attainable}

We first record a flexible version of the Walecki construction.

\begin{lemma}\label{lem:Sm}
Let $m\ge1$ and $L\ge m$. The elements of
\[
S_m(L)=\{1,2,\ldots,m-1,L\}
\]
can be arranged as a sequence of depth $m-1$.
\end{lemma}

\begin{proof}
Arrange them as
\[
L,1,m-1,2,m-2,3,\ldots.
\]
The consecutive absolute differences are
\[
L-1,m-2,m-3,\ldots,1,
\]
which form $S_{m-1}(L-1)$. Since $L-1\ge m-1$, the argument can be repeated. After $m-1$ steps one reaches a one-term sequence. No sequence of $m$ entries can have greater depth.
\end{proof}

Write $k=a(n)+t$, where $t\ge0$. The cases $n=0,1,2,3$ can be handled directly. For $n=0$ use $(1+t)$; for $n=1,2$ use, respectively,
\[
(1,2+t),\qquad (2,1,3+t).
\]
For $n=3$, the compositions obtained from
\[
(0,1,3,0),\quad (1,0,3,1),\quad (1,2,0,3+s)\quad(s=t-2\ge0)
\]
by adding $1$ to all four entries cover $t=0$, $t=1$, and $t\ge2$, respectively.

Assume now that $n\ge4$. In each residue class we use normalized sequences whose last three entries are
\[
0,n,1.
\]
If the last two entries are both increased by $s\ge0$, their sum increases by $2s$, while the difference set becomes
\[
\{1,2,\ldots,n-1,n+s\}=S_n(n+s).
\]
Lemma~\ref{lem:Sm} then gives at least $n-1$ further iterations after the first one.

The required base sequences are listed below. In each row, the displayed pattern continues in the evident period-$4$ manner and ends with $0,n,1$. Their differences are $1,\ldots,n$ and the indicated sums follow by direct calculation.

\begin{center}
\begin{tabular}{cclc}
\toprule
$n$ & base & normalized sequence & sum\\
\midrule
$4q+1$ & $Y_0$ & $(0,1,3,0,5,1,7,0,\ldots,n,1)$ & $M(n)$\\
        & $Y_1$ & $(0,2,3,0,5,1,7,0,\ldots,n,1)$ & $M(n)+1$\\[1mm]
$4q+2$ & $Y_0$ & $(0,2,1,4,0,6,1,8,0,\ldots,n,1)$ & $M(n)$\\
        & $Y_1$ & $(0,1,5,2,0,6,1,8,0,\ldots,n,1)$ & $M(n)+1$\\[1mm]
$4q$   & $Y_1$ & $(1,2,0,4,1,6,0,8,1,\ldots,n,1)$ & $M(n)+1$\\
        & $Y_2$ & $(3,1,0,4,1,6,0,8,1,\ldots,n,1)$ & $M(n)+2$\\[1mm]
$4q+3$ & $Y_1$ & $(1,0,3,1,5,0,7,1,9,0,\ldots,n,1)$ & $M(n)+1$\\
        & $Y_2$ & $(3,0,2,1,5,0,7,1,9,0,\ldots,n,1)$ & $M(n)+2$\\
\bottomrule
\end{tabular}
\end{center}

For $n\equiv1,2\pmod4$, use the base of sum $M(n)$ when $t$ is even and the base of sum $M(n)+1$ when $t$ is odd, then increase the last two entries by $\floor{t/2}$. For $n\equiv0,3\pmod4$, the case $t=0$ is supplied by the optimal construction of the preceding section. For $t>0$, use the base of sum $M(n)+1$ when $t$ is odd and the base of sum $M(n)+2$ when $t$ is even, and increase the final two entries by $(t-1)/2$ or $(t-2)/2$, respectively. Adding $1$ to all $n+1$ entries gives a composition of sum $a(n)+t=k$ and depth at least $n$.

Conversely, the lower-bound proof applies to every composition of depth at least $n$, so such a composition must have sum at least $a(n)$. This proves Theorem~\ref{thm:nogaps}.

\section{Classification and enumeration of minimizers}

Let $y=(y_0,\ldots,y_n)$ be normalized and minimizing. Equality in the lower-bound proof forces
\[
\{|y_i-y_{i-1}|:1\le i\le n\}=\{1,\ldots,n\}
\]
and
\[
B=2M(n)-\frac{n(n+1)}2=
\begin{cases}
n,&n\equiv0,1\pmod4,\\
n+1,&n\equiv2\pmod4,\\
n-1,&n\equiv3\pmod4.
\end{cases}
\]
Thus the total block excess
\[
\delta:=B-n
\]
belongs to $\{-1,0,1\}$; more precisely, $\delta=-1,0,1$ for $n\equiv3$, $n\equiv0,1$, and $n\equiv2\pmod4$, respectively.

A composition of $a(n)$ and depth $n$ must attain equality at every step of the lower-bound proof. Hence it has $n+1$ parts, its minimum part is $1$, and after subtracting $1$ its differences are exactly $1,\ldots,n$ and its sum is $M(n)$. Conversely, adding $1$ to any sequence counted by $N(n)$ gives a composition of $a(n)$ and depth $n$. This proves the first assertion of Theorem~\ref{thm:count}.

We now refine the block analysis used in Lemma~\ref{lem:Bbound}. The \emph{excess} of a block is its contribution to $B$ minus its number of edges.

\begin{lemma}[Local rigidity]\label{lem:rigidity}
In a minimizing sequence, every block is one of the following types, up to reversal:
\begin{enumerate}[label=\textup{(\roman*)}]
\item a deficient inner three-edge block
\[
(0,1,r+1,0),
\]
with difference set $\{1,r,r+1\}$ and excess $-1$;
\item an inner four-edge block
\[
(0,r+1,1,s+1,0),
\]
with difference set $\{r,r+1,s,s+1\}$ and excess $0$;
\item a one-edge end block $(1,0)$, of excess $0$;
\item a one-edge end block $(2,0)$, of excess $1$;
\item a two-edge end block
\[
(1,r+1,0),
\]
with difference set $\{r,r+1\}$ and excess $1$.
\end{enumerate}
At most one block of type \textup{(i)} can occur.
\end{lemma}

\begin{proof}
All inner blocks other than a deficient three-edge block have nonnegative excess, and all end blocks have nonnegative excess. A deficient block has excess $-1$ and contains difference $1$, so at most one can occur. Since the total excess is at most $1$, every individual block has excess at most $2$.

Consider an inner block
\[
(0,z_1,\ldots,z_{r-1},0),\qquad z_i>0,
\]
and let $b$ be its contribution to $B$. Then
\[
b=2\sum_{i=1}^{r-2}\min(z_i,z_{i+1})\ge2(r-2).
\]
The cases $r=1,2$ are impossible.

For $r=3$, the excess is
\[
2\min(z_1,z_2)-3.
\]
If the minimum is $1$, then, up to reversal, the block is
\[
(0,1,r+1,0),
\]
and has type~(i). If the minimum is $2$, the block has excess $1$; we temporarily call this an exceptional three-edge block. A larger minimum gives excess at least $3$.

For $r=4$, put
\[
a=\min(z_1,z_2),\qquad c=\min(z_2,z_3).
\]
The excess is $2(a+c)-4$. If $a+c=2$, then $a=c=1$. If $z_2>1$, necessarily $z_1=z_3=1$, and the two middle edge differences are equal. Hence $z_2=1$, which gives type~(ii). If $a+c=3$, then, up to reversal, $a=1$ and $c=2$. Thus $z_1=1$ and one of $z_2,z_3$ equals $2$. The choice $z_2=2$ repeats the first edge difference; therefore $z_3=2$ and the block has the form
\[
(0,1,u,2,0),\qquad u\ge5.
\]
It has excess $2$ and contains edge difference $1$. We call it an exceptional four-edge block. If $a+c\ge4$, the excess is at least $4$.

For $r=5$, excess at most $2$ would force $b=6$. Then
\[
\min(z_1,z_2)=\min(z_2,z_3)=\min(z_3,z_4)=1.
\]
No adjacent internal entries can both equal $1$, so the internal entries alternate as
\[
(1,u,1,v)\qquad\text{or}\qquad(u,1,v,1),
\]
with $u,v>1$. In the first case the two edges incident with $u$ have the same difference $u-1$; in the second case the two edges incident with $v$ have the same difference $v-1$. Both are impossible.

For $r=6$, excess at most $2$ would force equality in $b\ge8$, so every pair of consecutive internal entries would contain a $1$. Again the internal entries must alternate. They have one of the forms
\[
(1,u,1,v,1)\qquad\text{or}\qquad(u,1,v,1,w),
\]
and in either case two edges incident with the same internal entry have equal differences. Thus this case is impossible. For $r\ge7$, the lower bound gives excess at least $r-4\ge3$.

Now consider an end block
\[
(z_0,z_1,\ldots,z_{r-1},0),\qquad z_i>0.
\]
Its contribution is
\[
b=z_0+2\sum_{i=1}^{r-1}\min(z_{i-1},z_i)\ge2r-1.
\]
For $r=1$, excess at most $2$ allows $(1,0)$, $(2,0)$, and the temporary exceptional block $(3,0)$, of excesses $0,1,2$. For $r=2$, the only admissible value $b\le4$ is $b=3$, which forces type~(v); the alternative $b=4$ would give $(2,1,0)$ and repeat difference $1$. For $r=3$, excess at most $2$ would force $b=5$. Then
\[
z_0=1,\qquad \min(z_0,z_1)=1,\qquad \min(z_1,z_2)=1.
\]
Since $z_1\ne1$, we obtain $z_2=1$, and the first two edge differences are equal. For $r\ge4$, the excess is at least $r-1\ge3$.

It remains to exclude the three temporary exceptional types globally.

An exceptional four-edge block has excess $2$ and contains difference $1$. Since $\delta\le1$, it would have to coexist with a deficient block of excess $-1$, which also contains difference $1$. This is impossible.

If $(3,0)$ occurred, its excess $2$ would again force one deficient block and no other positive-excess block. The deficient block prevents any end block $(1,0)$, because both use difference $1$. Hence, for some $t\ge0$, the total number of edges would be
\[
n=1+3+4t=4(t+1),
\]
one edge from $(3,0)$, three from the deficient block, and $4t$ from ordinary inner blocks. Thus $n\equiv0\pmod4$, whereas the total excess $2-1=1$ requires $n\equiv2\pmod4$, a contradiction.

Finally suppose that $p\ge1$ exceptional three-edge blocks occur. Let $D\in\{0,1\}$ record the deficient block, and let $h_1,h_2\in\{0,1,2\}$ count positive-excess end blocks with one and two edges, respectively. Let $e$ count zero-excess one-edge end blocks. The excess equation and the edge count modulo $4$ are
\[
p+h_1+h_2-D=\delta,
\qquad
n\equiv3p+3D+h_1+2h_2+e\pmod4.
\]
If $D=1$, then $e=0$, because the deficient block and every zero-excess end block both use difference $1$. The complete residue check is:
\[
\begin{array}{c@{\qquad}c@{\qquad}c@{\qquad}c@{\qquad}c}
D & \delta & (p,h_1,h_2) & e & n\pmod4\\
\hline
0 & 1 & (1,0,0) & 0,1,2 & 3,0,1\\
1 & 0 & (1,0,0) & 0 & 2\\
1 & 1 & (2,0,0) & 0 & 1\\
1 & 1 & (1,1,0) & 0 & 3\\
1 & 1 & (1,0,1) & 0 & 0
\end{array}
\]
In the first row the required residue is $2$, since $\delta=1$. In the second row the required residue is $0$ or $1$, since $\delta=0$. In the last three rows the required residue is again $2$. No row is compatible. Finally, $\delta=-1$ would give $p+h_1+h_2=0$ when $D=1$, contradicting $p\ge1$, and is impossible when $D=0$. Thus no exceptional three-edge block occurs.

Only types~(i)--(v) remain.
\end{proof}

The edge counts and excesses in Lemma~\ref{lem:rigidity} leave only the global structures listed below. Each ordinary four-edge block uses two pairs of consecutive differences. Once the exceptional differences have been assigned, the remaining differences must be partitioned into consecutive pairs. Such a partition, when it exists, is unique: the smallest remaining difference must be paired with its successor, and induction completes the partition. This observation both forces the adjacent pairs displayed below and shows that the counting arguments neither omit nor multiply count any minimizing sequence.

\subsection*{$n=4q+3$}
Here the total excess is $-1$. Therefore one deficient three-edge block is necessary, and no block of positive excess can occur. The remaining $4q$ edges must consequently form exactly $q$ four-edge blocks, with no end blocks. The numbers $2,\ldots,4q+3$ are forced into the $2q+1$ adjacent pairs
\[
\{2,3\},\{4,5\},\ldots,\{4q+2,4q+3\}.
\]
Choose one pair for the deficient block, pair the remaining $2q$ pairs into $q$ unordered pairs, order the resulting $q+1$ blocks, and orient every block. The number of partitions of $2q$ labeled pairs into $q$ unordered two-element groups is
\[
\frac{(2q)!}{2^q q!}:
\]
order the $2q$ objects, split them into consecutive pairs, then divide by $2^q$ for the orders within the pairs and by $q!$ for the order of the groups. Hence
\[
N(4q+3)
=(2q+1)\frac{(2q)!}{2^q q!}(q+1)!2^{q+1}
=2(q+1)(2q+1)!.
\]

\subsection*{$n=4q$}
Here the total excess is $0$. Without a deficient block, all blocks have excess $0$; the edge count forces exactly $q$ four-edge blocks and no end blocks. Pairing
\[
\{1,2\},\{3,4\},\ldots,\{4q-1,4q\}
\]
into blocks, ordering, and orienting them gives
\[
(2q)!.
\]

With a deficient block of excess $-1$, exactly one block of excess $1$ is needed. A two-edge end block would leave an edge count congruent to $1$ modulo $4$, whereas a one-edge end block leaves a multiple of $4$. Hence the compensating block must be $(2,0)$ or its reversal, followed by $q-1$ four-edge blocks. After reserving differences $1$ and $2$, choose one of the $2q-1$ pairs
\[
\{3,4\},\ldots,\{4q-1,4q\}
\]
for the deficient block. Pair the rest, order and orient the inner blocks, and choose the end. This gives
\[
2(2q)!.
\]
Therefore
\[
N(4q)=3(2q)!.
\]

\subsection*{$n=4q+1$}
Here the total excess is $0$ and the edge count is odd. Without a deficient block, one excess-$0$ one-edge end block is necessary, and the remaining $4q$ edges form $q$ four-edge blocks. This structure gives
\[
2(2q)!.
\]

With a deficient block, one excess-$1$ block is required. A one-edge positive end block gives the wrong edge-count residue, so the positive block must be a two-edge end block, and the rest are $q-1$ four-edge blocks. Assigning two distinct pairs from
\[
\{2,3\},\{4,5\},\ldots,\{4q,4q+1\}
\]
to the deficient and end blocks, then grouping, ordering, orienting, and choosing the end gives
\[
4q(2q)!.
\]
For $q=0$ only the first structure occurs. Thus
\[
N(4q+1)=2(2q)!+4q(2q)!=2(2q+1)!.
\]

\subsection*{$n=4q+2$}
Here the total excess is $1$. Solving simultaneously for total excess $1$ and total edge count $4q+2$ gives exactly three possibilities:
\begin{enumerate}[label=\textup{(\alph*)}]
\item one two-edge end block and $q$ four-edge blocks;
\item two one-edge end blocks, with endpoint values $1$ and $2$, and $q$ four-edge blocks;
\item one deficient block, one one-edge end block of value $2$, one two-edge end block, and $q-1$ four-edge blocks.
\end{enumerate}
For (a), choose one pair from
\[
\{1,2\},\{3,4\},\ldots,\{4q+1,4q+2\}
\]
for the end block, then group, order, and orient the rest. This contributes $2(2q+1)!$. In (b), differences $1$ and $2$ are assigned to the two ends in either order; the remaining adjacent pairs contribute $2(2q)!$. In (c), differences $1$ and $2$ belong to the deficient and one-edge end blocks, while two distinct pairs from
\[
\{3,4\},\{5,6\},\ldots,\{4q+1,4q+2\}
\]
are assigned to the deficient and two-edge end blocks. This contributes $4q(2q)!$; for $q=0$ this structure is absent. Therefore
\[
N(4q+2)
=2(2q+1)!+2(2q)!+4q(2q)!
=4(2q+1)!.
\]
This proves Theorem~\ref{thm:count}.

\begin{remark}[Independent verification]\label{rem:verification}
A dynamic-programming enumeration, independent of the block classification, was used for $0\le n\le10$. It starts each walk at $0$, chooses every ordering of the edge lengths $1,\ldots,n$ and every sign assignment, and records the current endpoint, running minimum, and coordinate sum. Translation by the negative running minimum gives a unique normalized sequence, and conversely every normalized sequence determines a unique signed edge ordering. Thus no quotient by translations or reversals is taken in the count. The program counts exactly those translated walks whose coordinate sum is $M(n)$ and obtains
\[
N(n)=1,2,4,2,6,12,24,24,72,240,480,
\]
in agreement with Theorem~\ref{thm:count}. A direct brute-force implementation independently reproduces the same counts for $0\le n\le8$. Executable source code is included as the ancillary files \texttt{verify\_counts.py} and \texttt{verify\_counts\_bruteforce.py}.
\end{remark}

\section{Generating function and asymptotics}

The sequence $a(n)$ has ordinary generating function
\[
\sum_{n\ge0}a(n)x^n
=\frac{1+x+x^2-x^3+x^4-x^5}{(1-x)^2(1-x^4)}.
\]
Since
\[
a(n)=\frac{n^2}{4}+\frac{7n}{4}+O(1),
\]
Theorem~\ref{thm:nogaps} yields
\[
\boxed{d(k)=2\sqrt{k}-\frac72+O(1).}
\]
The error cannot in general be improved to $O(k^{-1/2})$, because $d(k)$ is constant on each interval $a(n)\le k<a(n+1)$, whose length is of order $n$. At the threshold values themselves one has the sharper relation
\[
n=2\sqrt{a(n)}-\frac72+O(n^{-1}).
\]

\section{Further questions}

Several natural variants remain open.
\begin{enumerate}
\item What is the maximum depth if the differences may not be permuted between iterations?
\item How many compositions of a general integer $k>a(n)$ attain the maximum depth $d(k)$?
\item What changes if repeated differences or zero differences are allowed?
\item Are there analogous formulas for cyclic consecutive differences?
\end{enumerate}

\section*{Code availability}
The dynamic-programming verification code, the independent brute-force count, the construction audit, the block-structure audit, and their recorded outputs are included with the arXiv source as ancillary files. The proofs of the three main theorems do not depend on these computations.

\section*{Acknowledgments}
The author used computational and language-model tools during the exploratory and editorial phases of this work. All mathematical statements, proofs, and conclusions remain the responsibility of the author.

\end{document}